\documentclass[10pt]{article}
\usepackage[english]{babel} % language
\usepackage[margin=1.5in]{geometry} % geometry of page
\usepackage{amsmath, amsfonts, amssymb, amsthm} % math
\usepackage[hidelinks]{hyperref} % cross referencing
\usepackage[scr=boondoxo,cal=euler]{mathalfa} % mathscr and mathcal fonts
\usepackage{tikz-cd} % commutative diagrams
\usepackage{pxfonts} % font

\usepackage{indentfirst} % indent first paragraph
\usepackage{fancyhdr} % header and footer
\usepackage{graphicx} % for includegraphics
\usepackage{enumitem} % better enumerate, itemize
\usepackage{microtype} % better kerning
\usepackage{pdfpages} % insert pdf documents
\usepackage{centernot} % slash through symbols
\usepackage{ragged2e} % for \RightRagged 
\usepackage{mathabx}
\usepackage{pinlabel}
\usepackage{caption}

\allowdisplaybreaks % allow multiple page display math

\usetikzlibrary{decorations.pathmorphing}

\theoremstyle{plain}
\newtheorem{thm}{Theorem}[section]
\newtheorem*{thm*}{Theorem}
\newtheorem{lem}[thm]{Lemma}
\newtheorem*{lem*}{Lemma}
\newtheorem{cor}[thm]{Corollary}
\newtheorem*{cor*}{Corollary}
\newtheorem{prop}[thm]{Proposition}
\newtheorem*{prop*}{Proposition}

\newtheorem*{ques*}{Question}

\theoremstyle{definition}
\newtheorem{df}[thm]{Definition}
\newtheorem*{df*}{Definition}
\newtheorem*{dfs*}{Definitions}

\newtheorem*{exercise*}{Exercise}

\theoremstyle{remark}
\newtheorem{rem}[thm]{Remark}
\newtheorem*{rem*}{Remark}

\usepackage{etoolbox}
\patchcmd{\thmhead}{(#3)}{#3}{}{}
\makeatletter
\g@addto@macro\bfseries{\boldmath}
\makeatother

\newcommand{\fk}[1]{\mathfrak{#1}}

\newcommand{\sr}[1]{\mathscr{#1}}
\newcommand{\bo}[1]{\boldsymbol{#1}} % bold math

\newcommand{\Z}{\mathbf{Z}} % integers
\newcommand{\Q}{\mathbf{Q}} % rationals
\newcommand{\C}{\mathbf{C}} % complex numbers
\newcommand{\G}{\mathbf{G}} % Grassmannian
\newcommand{\F}{\mathbf{F}}
\newcommand{\CP}{\mathbf{CP}} % complex projective space
\newcommand{\x}{\times}

\renewcommand{\sl}{\fk{sl}}

\newcommand{\HFKhat}{\smash{\widehat{\mathrm{HFK}}}}

\newcommand{\rKh}{\smash{\overline{\Kh}}}
\newcommand{\rKR}{\smash{\overline{\KR}}}

\DeclareMathOperator{\rk}{rk}

\DeclareMathOperator{\Int}{Int}

\DeclareMathOperator{\Image}{Im}

\DeclareMathOperator{\Id}{Id}

\DeclareMathOperator{\Spin}{Spin}

\DeclareMathOperator{\Kh}{Kh}

\DeclareMathOperator{\KR}{KR}
\DeclareMathOperator{\KRC}{KRC}

\DeclareMathOperator{\SFH}{SFH}
\DeclareMathOperator{\SFC}{SFC}

\title{Split link detection for\\$\sl(P)$ link homology in characteristic $P$}
\author{Joshua Wang}
\date{}
\begin{document}
\maketitle
	
\begin{abstract}
	If $P$ is a prime number, we show that reduced $\sl(P)$ link homology with coefficients in $\Z/P$ detects split links. The argument uses Dowlin's spectral sequence and sutured Floer homology with twisted coefficients. When $P = 2$, we recover Lipshitz--Sarkar's split link detection result for Khovanov homology with $\Z/2$ coefficients. 
\end{abstract}

\section{Introduction}

Khovanov homology \cite{MR1740682} strikes a nice balance of computability and power. Although it is defined in an elementary and combinatorial way, it detects the unknot \cite{MR2805599}, the trefoils \cite{baldwin2021khovanovtrefoil}, the figure-eight \cite{MR4275096}, and the cinquefoil \cite{baldwin2021khovanovcinquefoil}. Additionally, it detects many links \cite{MR2653731,MR3190305,MR3332892,MR4049809,martin2020khovanov,xie2020links,li2020detection}, and the complete list of links with minimal rank Khovanov homology is known \cite{xie2021classification}. Beyond the detection of individual knots and links, Khovanov homology also detects certain topological properties. It detects whether a link is split \cite{lipshitz2019khovanov}, and it detects the connected sum of a fixed pair of knots among all band sums of the pair \cite{wang2020cosmetic}. All detection results so far have been derived from relationships between Khovanov homology and invariants arising from Floer theory, such as instanton homology \cite{MR2805599} and Heegaard Floer homology \cite{MR2141852}. 

Khovanov homology is the $N = 2$ case of $\sl(N)$ link homology \cite{MR2391017}, which can also be defined combinatorially \cite{MR4164001}, but is much less well-understood. Using a spectral sequence relating $\sl(N)$ link homology and Khovanov homology \cite{MR3982970}, some detection results for $\sl(N)$ link homology can be obtained from the analogous detection results for Khovanov homology. In this paper, we prove a detection result for $\sl(N)$ link homology that does not appear to follow from the analogous detection result for Khovanov homology. 
We prove that when $P$ is prime, $\sl(P)$ link homology with coefficients in the field $\F_P \coloneq \Z/P$ detects split links, generalizing and recovering Lipshitz--Sarkar's split link detection result for Khovanov homology with $\F_2$ coefficients \cite[Theorem 1]{lipshitz2019khovanov}. 

We let $\rKR_N(L,q;R)$ denote the reduced $\sl(N)$ link homology of a link $L$ with respect to a basepoint $q \in L$, where coefficients are taken in a ring $R$. An additional basepoint $r\in L$ defines a basepoint operator $X_r\colon \rKR_N(L,q;R) \to \rKR_N(L,q;R)$ satisfying $X_r^N = 0$. 

\begin{thm}\label{thm:mainthm}
	Let $L$ be an oriented link in $S^3$ with basepoints $q,r \in L$, and let $P$ be a prime number. There is an embedded sphere in $S^3\setminus L$ that separates $q$ and $r$ if and only if $\rKR_P(L,q;\F_P)$ is a free module over $\F_P[X]/X^P$ where the action of $X$ is the basepoint operator $X_r$.
\end{thm}

\begin{rem}
	The case when $P = 2$ is \cite[Theorem 1]{lipshitz2019khovanov}. Although our proof of Theorem~\ref{thm:mainthm} uses the spectral sequence relating $\sl(P)$ link homology and Khovanov homology, our method is indirect and requires proving a new detection result for Khovanov homology using Floer theory. 
\end{rem}

The restriction on the coefficient field and on the primality of $P$ arises only in the proof that reduced $\sl(P)$ link homology is free over $\F_P[X]/X^P$ when $L$ is split. This implication is equivalent to a general statement about the relationship between reduced and unreduced $\sl(P)$ link homology in characteristic $P$ that the author proves in \cite{wang2021sln}. The reverse implication does not require these restrictions. 

\begin{thm}\label{thm:freeImpliesSplit}
	Let $L$ be an oriented link in $S^3$ with basepoints $q,r \in L$, let $N \ge 2$, and let $\F$ be a field. If $\rKR_N(L,q;\F)$ is a free module over $\F[X]/X^N$ where $X = X_r$, then there is an embedded sphere in $S^3\setminus L$ that separates $q$ and $r$. 
\end{thm}

\begin{rem}
	In particular, if the reduced Khovanov homology $\rKh(L,q;\F)$ of $L$ with respect to $q$ is free over $\F[X]/X^2$ where $X = X_r$, there is an embedded sphere in $S^3\setminus L$ that separates $q$ and $r$. 
	Lipshitz--Sarkar's proof of this statement when $\F = \F_2$ does not appear to generalize to arbitrary fields since they use the spectral sequence from reduced Khovanov homology to the Heegaard Floer homology of the double branched cover \cite{MR2141852} which requires $\F_2$ coefficients in an essential way \cite{MR3071132}. 
\end{rem}

\begin{rem}
	We note that the condition that $\rKR_N(L,q;\F)$ is free over $\F[X]/X^N$ where $X = X_r$ is not symmetric in $q$ and $r$. For example, if $\F = \Q$ and $L$ is the split union of the trefoil $T_{2,3}$ and the unknot $U$, then the condition is satisfied if $q \in T_{2,3}$ and $r \in U$ but fails if $r \in T_{2,3}$ and $q \in U$. If $N = P$ is prime and $\F = \F_P$, the condition is symmetric by \cite[Theorem 1.1]{wang2021sln}. 
\end{rem}

We prove Theorem~\ref{thm:freeImpliesSplit} by relating the module structure of $\rKR_N(L,q;\F)$ to the $\sl(N)$ link homology groups of band surgeries on $L$. Assume that $q$ and $r$ lie on distinct components of $L$, and let $b$ be an orientation-preserving band that merges the components of $L$ marked by $q$ and $r$. Let $K_b$ be the link obtained by band surgery along $b$. Note that $K_b$ has one fewer component than $L$. Then let $K_{b+n}$ be obtained by adding $n \in \Z$ full twists to the band. We assume that the band surgery is done away from the basepoints $q,r$ so that they can be viewed as basepoints on $K_{b+n}$. See Figure~\ref{fig:bandSurgeries}. 

\begin{figure}[!ht]
	\centering
	\vspace{5pt}
	\labellist
		\pinlabel $K_b$ at 510 -15
		\pinlabel $K_{b+1}$ at 880 -15
		\pinlabel $b$ at 50 245
		\pinlabel $L$ at 145 -15
		\pinlabel $\bullet$ at 27 27
		\pinlabel $q$ at 13 10
		\pinlabel $\bullet$ at 269 27
		\pinlabel $r$ at 278 12
		\pinlabel $\bullet$ at 386 27
		\pinlabel $q$ at 372 10
		\pinlabel $\bullet$ at 628 27
		\pinlabel $r$ at 637 12
		\pinlabel $\bullet$ at 745 27
		\pinlabel $q$ at 731 10
		\pinlabel $\bullet$ at 987 27
		\pinlabel $r$ at 996 12
	\endlabellist
	\includegraphics[width=.27\textwidth]{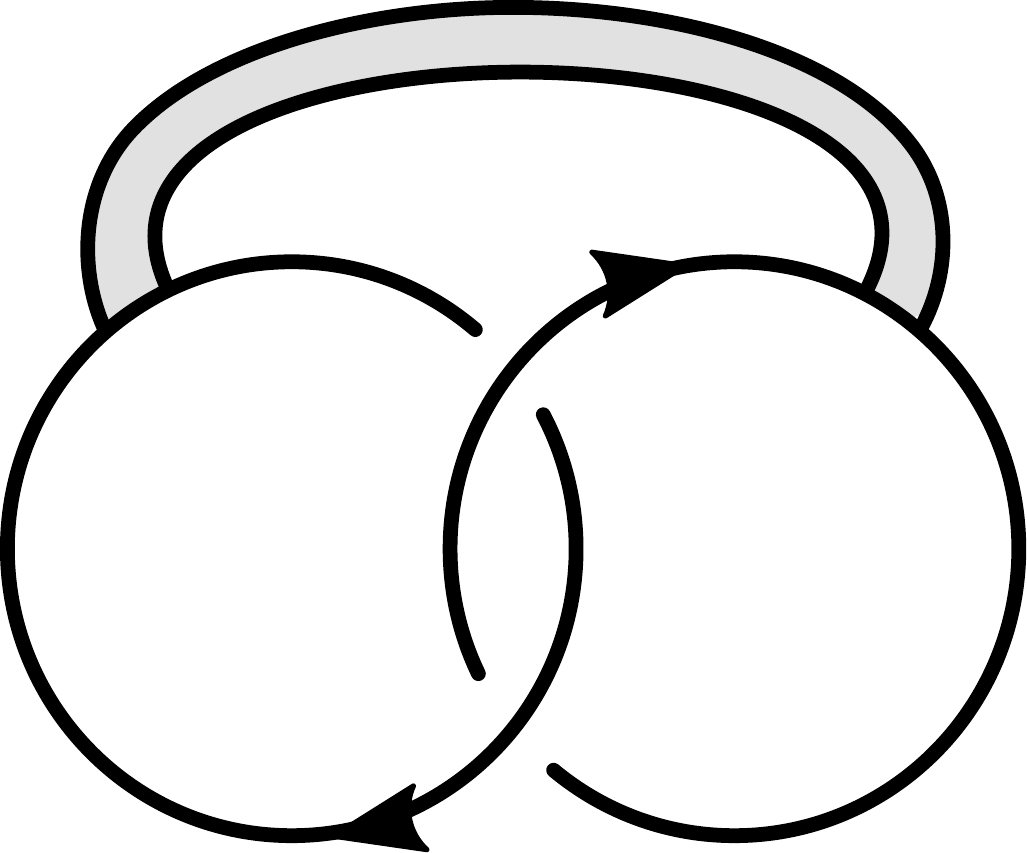}\hspace{20pt}
	\includegraphics[width=.27\textwidth]{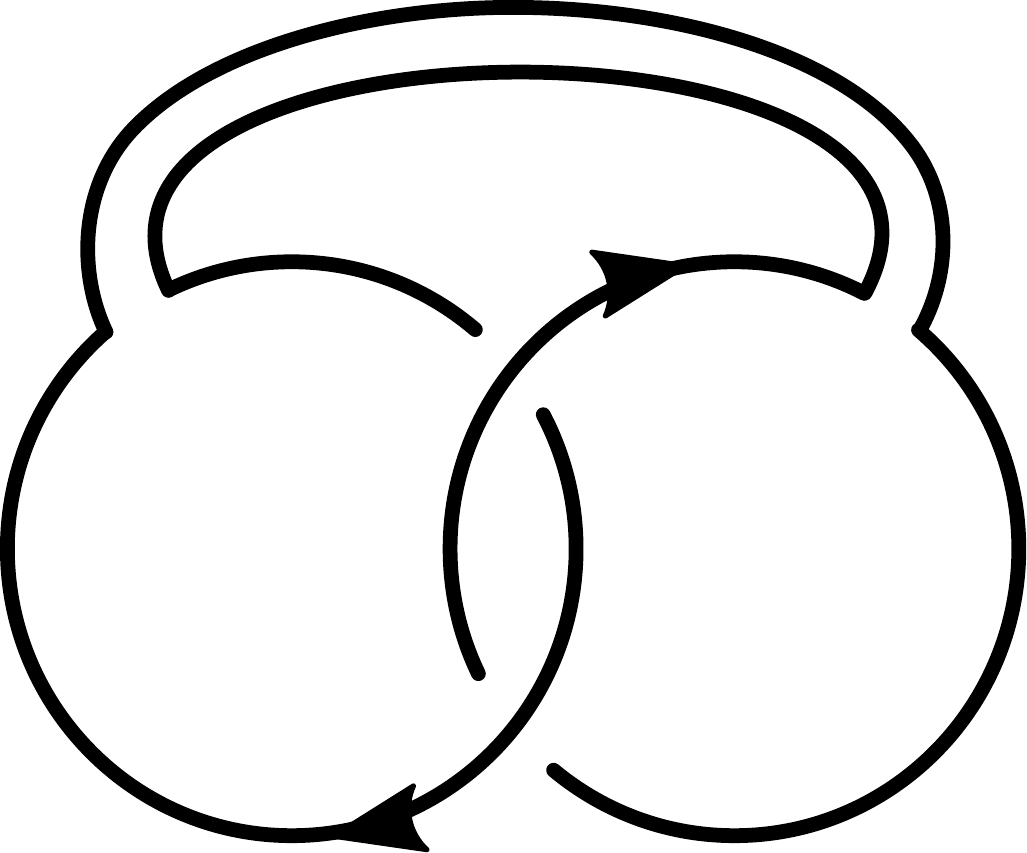}\hspace{20pt}
	\includegraphics[width=.27\textwidth]{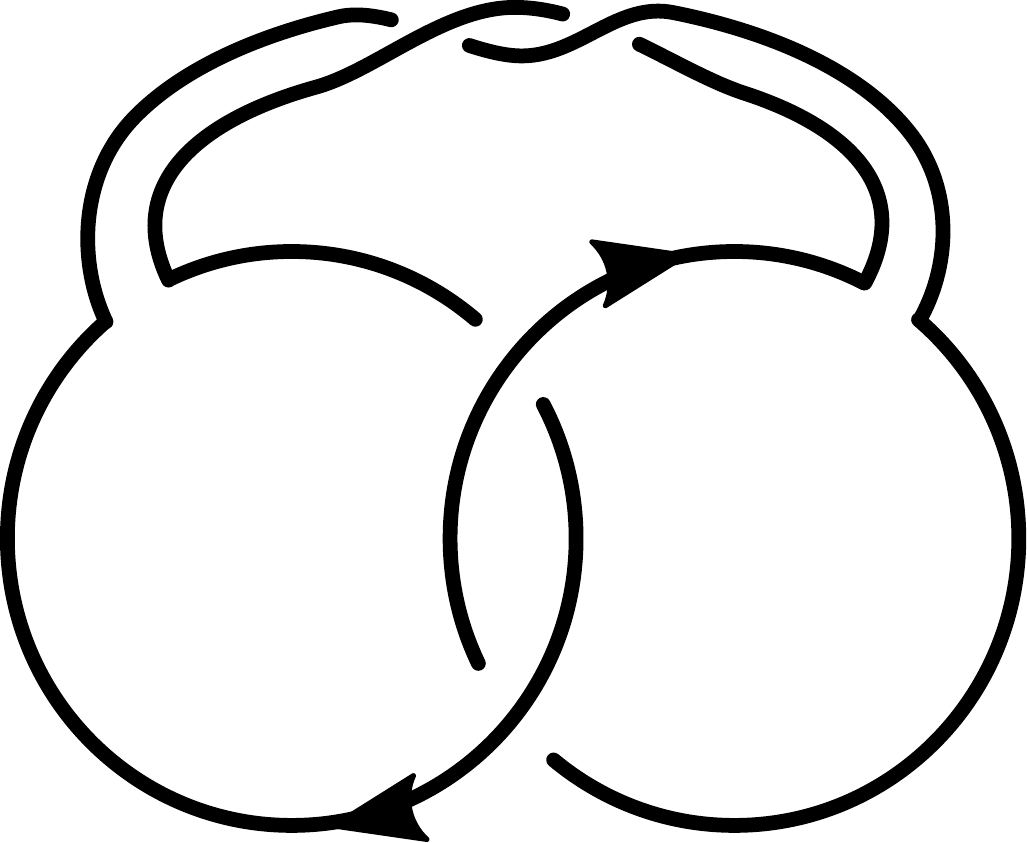}
	\vspace{10pt}
	\captionsetup{width=.8\linewidth}
	\caption{$K_b$ is obtained from $L$ by band surgery along a band $b$. $K_{b+n}$ is then obtained by adding $n$ full twists.}
	\label{fig:bandSurgeries}
\end{figure}
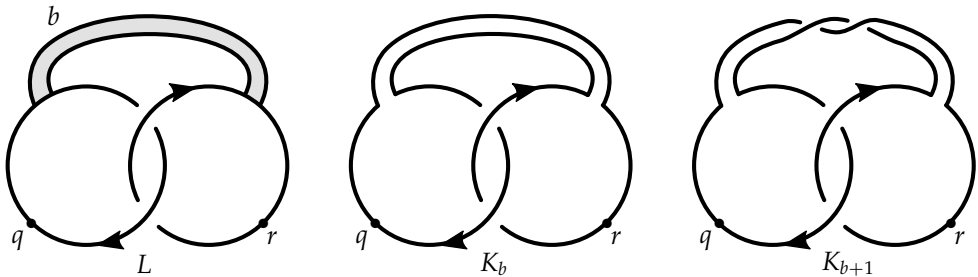

The following two propositions imply Theorem~\ref{thm:freeImpliesSplit}. In both propositions, $N \ge 2$ is an integer, $\F$ is any field, and $L$ is an oriented link with basepoints $q,r$ on distinct components.

\begin{prop}\label{prop:freeImpliesDimIndependent}
	If $\rKR_N(L,q;\F)$ is a free module over $\F[X]/X^N$ where $X = X_r$, then for any band $b$ that merges the marked components of $L$, the dimension of $\rKR_N(K_{b+n},q;\F)$ is independent of $n\in\Z$. 
\end{prop}

\begin{prop}\label{prop:threeEquivalentStatements}
	The following three statements are equivalent: \begin{enumerate}[noitemsep]
		\item There is an embedded sphere in $S^3\setminus L$ that separates $q$ and $r$. 
		\item For any band $b$ that merges the marked components of $L$, the dimension of $\rKR_N(K_{b+n},q;\F)$ is independent of $n \in \Z$. 
		\item For some band $b$ that merges the marked components of $L$, the dimension of $\rKR_N(K_{b+n},q;\F)$ is a bounded function of $n \in \Z$.
	\end{enumerate}
\end{prop}

\begin{rem}
	In fact, if there does not exist an embedded sphere in $S^3\setminus L$ that separates $q$ and $r$, \[
		\dim \rKR_N(K_{b+n},q;\F) \to \infty
	\]as $n \to \infty$ and as $n \to -\infty$ for any band $b$ that merges the marked components of $L$. 
\end{rem}
\begin{rem}
	We prove the analog of Proposition~\ref{prop:threeEquivalentStatements} for knot Floer homology in Theorem~\ref{thm:unboundedHFKhat}.
\end{rem}

Proposition~\ref{prop:freeImpliesDimIndependent} is proved using a skein exact triangle argument. The first statement of Proposition~\ref{prop:threeEquivalentStatements} implies the second by a ribbon concordance argument that the author previously used in connection with the cosmetic crossing conjecture \cite{wang2020cosmetic}. The third statement implies that the dimension of $\rKh(K_{b+n},q;\Q)$ is a bounded function of $n \in \Z$ by the universal coefficient theorem and the spectral sequence relating $\sl(N)$ link homology and Khovanov homology \cite{MR3982970}. By Dowlin's spectral sequence \cite{dowlin2018spectral}, the dimension of $\HFKhat(K_{b+n};\Q)$ are also bounded in $n$. By a more elaborate version of the main argument of \cite{MR4311818} using sutured Floer homology with twisted coefficients, we deduce the existence of an appropriate sphere in $S^3\setminus L$.

\theoremstyle{definition}
\newtheorem*{ack}{Acknowledgments}
\begin{ack}
	I thank Artem Kotelskiy and Claudius Zibrowius for many helpful discussions, and I thank Paul Wedrich for explaining to me the rank-reducing spectral sequence for links. I thank Robert Lipshitz for useful feedback and suggestions, especially regarding orientations. 
	I also thank my advisor Peter Kronheimer for his continued guidance, support, and encouragement. This material is based upon work supported by the NSF GRFP through grant DGE-1745303.
\end{ack}

\section{Sutured Floer homology with twisted coefficients}

The purpose of this section is to prove the following result. Only the statement of this result will be used in the proofs of the main results in the next section. 

\begin{thm}\label{thm:unboundedHFKhat}
	Let $L$ be an oriented link in $S^3$ with basepoints $q,r \in L$ on distinct components. 
	If there does not exist an embedded sphere in $S^3\setminus L$ that separates $q$ and $r$, then for any band $b$ merging the marked components of $L$,\[
		\dim_\Q \HFKhat(K_{b + n};\Q) \to \infty \quad\text{ as }n \to \infty\text{ and as }n \to -\infty.
	\]
\end{thm}

\begin{rem}
	If there does exist a sphere in $S^3\setminus L$ that separates $q$ and $r$, then $\HFKhat(K_{b+n};\Q)$ is grading-preserving isomorphic to $\HFKhat(K_b;\Q)$ for all $n \in \Z$ by the proof of \cite[Theorem 1.3]{wang2020cosmetic}.
\end{rem}

In sections~\ref{subsec:suturedPrelims}, \ref{subsec:surfaceDecomp}, and \ref{subsec:surgTriangle}, we explain a few basic features of sutured Floer theory with twisted coefficients that we use in section~\ref{subsec:bandSurgeries} to prove Theorem~\ref{thm:unboundedHFKhat}. In particular, section~\ref{subsec:surfaceDecomp} provides a direct summand relationship under nice surface decomposition and section~\ref{subsec:surgTriangle} provides a surgery exact triangle.

\subsection{Preliminaries}\label{subsec:suturedPrelims}

See \cite{Gab83,Gab87a,Juh06,Juh08,MR2653728} for the definitions of balanced sutured manifolds, nice surface decompositions, and sutured Floer homology. 
The basic features of sutured Floer theory with twisted coefficients follow from straightforward adaptations of the usual constructions and arguments in sutured Floer homology and Heegaard Floer homology. 
We work over the field of rational numbers $\Q$ because Dowlin's spectral sequence \cite{dowlin2018spectral} is defined over $\Q$, and we only consider the case of coefficients in $\Q[\Z/n] = \Q[T]/(T^n - 1)$ where the ``twisting'' is with respect to a relative first homology class of the sutured manifold. 

The following definition is a straightforward generalization of twisted coefficients in Heegaard Floer homology \cite[Section 8]{MR2113020}. It is basically a special case of the construction in \cite[Section 4.1]{juhasz2020transverse}. Also see \cite[Section 3]{MR2377279} and \cite[Section 2.1]{MR3294567}. For systems of orientations for sutured Floer homology, see \cite[section 5]{MR3412088}. 

\begin{df}
	Let $(\Sigma,\bo\alpha,\bo\beta)$ be an admissible balanced diagram for a balanced sutured manifold $(M,\gamma)$, and let $\omega = \sum k_i c_i$ be a finite formal sum of properly embedded oriented curves $c_i$ on $\Sigma$ with integer coefficients $k_i$. Each $c_i$ is required to intersect the $\alpha$- and $\beta$-curves transversely and to be disjoint from every intersection point of the $\alpha$- and $\beta$-curves. Note that $\omega$ represents a class $\zeta = [\omega] \in H_1(M,\partial M)$, and every class in $H_1(M,\partial M)$ is represented by such a relative $1$-cycle on $\Sigma$. 

	Let $\underline{\SFC}(\Sigma,\bo\alpha,\bo\beta;\Q[\Z/n]_\omega)$ be the free $\Q[T]/(T^n - 1)$-module with basis $\mathbf{T}_\alpha \cap \mathbf{T}_\beta$ equipped with the $T$-equivariant differential \[
		\partial \mathbf{x} = \sum_{\mathbf{y} \in \mathbf{T}_\alpha \cap \mathbf{T}_\beta} \sum_{\substack{\phi \in \pi_2(\mathbf{x},\mathbf{y})\\\mu(\phi) = 1}} \# \widehat{\sr M}(\phi) \:T^{\:\omega\cdot\partial_\beta D(\phi)} \cdot \mathbf{y}
	\]for $\mathbf{x} \in \mathbf{T}_\beta \cap \mathbf{T}_\beta$ defined with respect to a suitable family of almost complex structures. The integer $\omega \cdot \partial_\beta D(\phi)$ is the algebraic intersection number between $\omega$ and an oriented multi-arc $\partial_\beta D(\phi)$, which is defined as follows. A Whitney disc $\phi \in \pi_2(\mathbf{x},\mathbf{y})$ has an associated $2$-chain on $\Sigma$ called its domain $D(\phi)$, and $\partial_\beta D(\phi)$ is defined to be the part of $\partial D(\phi)$ lying in the $\beta$-circles, thought of as an oriented multi-arc from $\mathbf{x}$ to $\mathbf{y}$. The count $\# \smash{\widehat{\sr M}}(\phi)$ is taken with signs with respect to a system of orientations $\fk{o}$ which we suppress from notation. 

	Let $\underline{\SFH}(M,\gamma;\Q[\Z/n]_\zeta)$ denote the homology of $\underline{\SFC}(\Sigma,\bo\alpha,\bo\beta;\Q[\Z/n]_\omega)$. In this paper, we view this homology group as simply a vector space over $\Q$ without additional structure. Up to isomorphism, this vector space depends only on $(M,\gamma)$, the class $\zeta = [\omega] \in H_1(M,\partial M)$, the integer $n \ge 1$, and the weak equivalence class \cite[section 2]{MR2812456} of $\fk{o}$. 
\end{df}
\begin{rem}
	In similar contexts, the formula $\omega\cdot\partial_\alpha D(\phi)$ is used instead of $\omega\cdot \partial_\beta D(\phi)$, though the choice is inconsequential. We use $\omega\cdot \partial_\beta D(\phi)$ to match the conventions in the proof of \cite[Theorem 3.1]{MR2377279} when establishing the surgery exact triangle in Section~\ref{subsec:surgTriangle}.
\end{rem}

\begin{rem}\label{rem:trivialCasesTwistedCoeff}
	If $n = 1$, note that $\underline{\SFH}(M,\gamma;\Q[\Z/n]_\zeta) = \SFH(M,\gamma;\Q)$, which is independent of $\zeta$. Also note that if $\zeta = 0$, then $\underline{\SFH}(M,\gamma;\Q[\Z/n]_\zeta) \cong \bigoplus^n \SFH(M,\gamma;\Q)$. 
\end{rem}

\subsection{Surface decompositions}\label{subsec:surfaceDecomp}

The next proposition is a straightforward adaptation of \cite[Theorem 1.1]{MR3294567} which itself is a generalization of \cite[Theorem 1.3]{Juh08}. We highlight the main ingredients in the proof. 

\begin{prop}\label{prop:twistedDirectSummandSurfDecomp}
	Let $(M,\gamma) \overset{S}{\rightsquigarrow} (M',\gamma')$ be a nice surface decomposition of balanced sutured manifolds. Let $i_*\colon H_1(M,\partial M) \to H_1(M, (\partial M) \cup S) \cong H_1(M',\partial M')$ be the map induced by the inclusion map $i\colon (M,\partial M) \to (M,(\partial M) \cup S)$. Then $\underline{\SFH}(M',\gamma';\Q[\Z/n]_{i_*(\zeta)})$, defined with respect to an induced system of orientations, is a direct summand of $\underline{\SFH}(M,\gamma;\Q[\Z/n]_\zeta)$ for any $\zeta \in H_1(M,\partial M)$. 
\end{prop}

\begin{proof}[Proof sketch]
	The surface may be assumed to be a good decomposing surface, and there exists a nice admissible good surface diagram adapted to it \cite[proof of Theorem 1.3]{Juh08}. In particular, we are provided with admissible balanced diagrams for $(M,\gamma)$ and $(M',\gamma')$ and an identification between the sutured Floer complex for $(M',\gamma')$ with a direct summand of the sutured Floer complex for $(M,\gamma)$. The identification is at the level of generators and discs with orientations. By \cite[Lemma 2.6]{MR3294567}, a relative $1$-cycle $\omega$ representing $\zeta$ can be chosen so that it induces a relative $1$-cycle $\omega'$ representing $i_*(\zeta)$ in such a way that the above identification of discs intertwines algebraic intersection number with $\omega$ and $\omega'$. 
\end{proof}

Similarly, the analogous adaption of \cite[Lemma 9.13]{Juh06} is the following result. 

\begin{prop}\label{prop:productDiscDecomp}
	Let $(M,\gamma) \overset{D}{\rightsquigarrow} (M',\gamma')$ be a product disc decomposition of balanced sutured manifolds. Then the direct summand relationship of Proposition~\ref{prop:twistedDirectSummandSurfDecomp} is an isomorphism $\underline{\SFH}(M,\gamma;\Q[\Z/n]_\zeta) \cong \underline{\SFH}(M',\gamma';\Q[\Z/n]_{i_*(\zeta)})$ for any $\zeta \in H_1(M,\partial M)$.
\end{prop}

\begin{lem}\label{lem:twistedProductSutMfldComputation}
	Let $(M,\gamma)$ be a balanced product sutured manifold, and let $\zeta \in H_1(M,\partial M)$. Then for any system of orientations, \[
		\dim_\Q \underline{\SFH}(M,\gamma;\Q[\Z/n]_\zeta) = n.
	\]
\end{lem}
\begin{proof}
	Following \cite[Proposition 9.4]{Juh06}, there is an admissible balanced diagram $(\Sigma,\bo\alpha,\bo\beta)$ for $(M,\gamma)$ for which $|\textbf{T}_\alpha \cap \textbf{T}_\beta| = 1$, so $\dim_\Q \underline{\SFH}(M,\gamma;\Q[\Z/n]_\zeta) = \dim_\Q \Q[\Z/n] = n$. 
\end{proof}

\begin{cor}\label{cor:dimInequalTaut}
	Let $(M,\gamma)$ be a taut balanced sutured manifold, and let $\zeta \in H_1(M,\partial M)$. Then for any system of orientations,\[
		\dim_\Q \underline{\SFH}(M,\gamma;\Q[\Z/n]_\zeta) \ge n.
	\]
\end{cor}
\begin{proof}
	Just as in \cite[Theorem 1.4]{Juh08}, there is a sequence of nice surface decompositions\[
		(M,\gamma) \overset{S_1}{\rightsquigarrow} (M_1,\gamma_1) \overset{S_2}{\rightsquigarrow} \cdots \overset{S_n}{\rightsquigarrow} (M_m,\gamma_m)
	\]where $(M_m,\gamma_m)$ is a balanced product sutured manifold by \cite[Theorem 4.2]{Gab83} and \cite[Theorem 8.2]{Juh08}. The result then follows from Proposition~\ref{prop:twistedDirectSummandSurfDecomp} and Lemma~\ref{lem:twistedProductSutMfldComputation}. 
\end{proof}
\begin{rem}
	The conclusion of Corollary~\ref{cor:dimInequalTaut} can fail when $(M,\gamma)$ is not taut. For example, if $\zeta$ is a generator of the relative first homology group of $S^3(2)$, then by direct computation, $\dim_\Q \underline{\SFH}(S^3(2);\Q[\Z/n]_\zeta) = 2$ for all $n \ge 1$ for the standard choice of system of orientations. 
\end{rem}

\subsection{A surgery exact triangle}\label{subsec:surgTriangle}

Let $(M,\gamma)$ be a balanced sutured manifold, and let $K \subset \Int(M)$ be a knot. An \textit{integer}-framing $\lambda$ of $K$ is a framing that intersects the meridian $\mu$ of $K$ once. Let $M_\lambda(K)$ and $M_{\lambda+n\mu}(K)$ denote the $3$-manifolds obtained by Dehn surgery along $K$ with integer-framings $\lambda$ and $\lambda + n\mu$, respectively. Note that $(M_\lambda(K),\gamma)$ and $(M_{\lambda+n\mu}(K),\gamma)$ are balanced sutured manifolds, using the identifications $\partial M = \partial M_\lambda(K) = \partial M_{\lambda+n\mu}(K)$. The following result is an adaptation of \cite[Theorem 3.1]{MR2377279}. Remark~\ref{rem:relationshipSurgeryTriangles} clarifies the relationship between these two results.

\begin{prop}\label{prop:surgeryTriangleTwisted}
	Let $K$ be a knot in a balanced sutured manifold $(M,\gamma)$ with integer-framing $\lambda$. For any $n \ge 1$ and any systems of orientations, there is an exact triangle \[
		\begin{tikzcd}[column sep = -3em]
			\SFH(M_\lambda(K),\gamma;\Q) \ar[rr] & & \SFH(M_{\lambda+n\mu}(K),\gamma;\Q) \ar[dl]\\
			& \underline{\SFH}(M,\gamma;\Q[\Z/n]_\zeta) \ar[ul] &
		\end{tikzcd}
	\]where $\zeta \in H_1(M,\partial M)$ is the class represented by $K$. 
\end{prop}

\begin{rem}\label{rem:relationshipSurgeryTriangles}
	Proposition~\ref{prop:surgeryTriangleTwisted} generalizes the statement of \cite[Theorem 3.1]{MR2377279} in two ways: the knot $K$ is no longer required to be nullhomologous, and the ambient $3$-manifold is an arbitrary balanced sutured manifold rather than a closed oriented $3$-manifold. Note though that Proposition~\ref{prop:surgeryTriangleTwisted} is for the ``hat'' version of Heegaard Floer homology, whereas \cite[Theorem 3.1]{MR2377279} is for the ``plus'' version. 

	Relaxing the condition that $K$ is nullhomologous is not actually a generalization of what Ozsv\'ath--Szab\'o proved. Their proof of \cite[Theorem 3.1]{MR2377279} may be viewed as having two steps. They first establish an exact triangle where one term has twisted coefficients without restriction on the homology class of $K$, which is essentially the triangle in Proposition~\ref{prop:surgeryTriangleTwisted}. Then, under the assumption that $K$ is nullhomologous, they compute this twisted coefficient group in terms of Floer homology groups with ordinary coefficients (also see Remark~\ref{rem:trivialCasesTwistedCoeff}). 

	Generalizing surgery triangles from closed $3$-manifolds to balanced sutured manifolds is standard and appears in the literature. For examples of such generalizations to sutured Floer homology, see \cite[Section 5.1]{MR3617634} for the surgery triangle and \cite[Sections 3-4]{MR2601010} for the link surgeries spectral sequence. For these reasons, we only give a sketch of the proof, with emphasis on the places were the proof differs, and refer to \cite{MR2377279,MR2601010} for more details. For the assertion about systems of orientations, see \cite[section 6.1, Lemma 6.6]{MR3412088}. See also \cite{MR2249248} for an exposition of the exact triangle in Heegaard Floer homology, which includes the triangle detection lemma \cite[Lemma 2.13]{MR2249248}. 
\end{rem}

\begin{proof}[Proof outline of Proposition~\ref{prop:surgeryTriangleTwisted}]
	There is a compact oriented surface $\Sigma$ together with four ordered sets $\bo\alpha,\bo\beta,\bo\gamma,\bo\delta$, each consisting of $k$ disjoint simple closed curves on $\Sigma$ such that \begin{itemize}[noitemsep]
		\item $(\Sigma,\bo\alpha,\bo\beta)$, $(\Sigma,\bo\alpha,\bo\gamma)$, $(\Sigma,\bo\alpha,\bo\delta)$ are balanced diagrams that represent $(M,\gamma)$, $(M_\lambda(K),\gamma)$, $(M_{\lambda+n\mu}(K),\gamma)$, respectively, 
		\item for $i = 1,\ldots,k-1$, the curves $\beta_i,\gamma_i$, $\delta_i$ are small isotopic translates, pairwise intersecting in two points transversely, and are disjoint from the other curves in $\bo\beta,\bo\gamma,\bo\delta$,
		\item there is a torus summand of $\Sigma$ (an embedding $T^2\setminus D^2 \hookrightarrow \Sigma$) which contains $\beta_k, \gamma_k, \delta_k$ and is disjoint from $\beta_i, \gamma_i, \delta_i$ for $i = 1,\ldots,k-1$ where \begin{itemize}[noitemsep]
			\item $\beta_k$ and $\gamma_k$ are coordinate factors on $T^2$, 
			\item $\delta_k$ represents the slope $\gamma_k + n\beta_k$ where $\gamma_k$ and $\beta_k$ are oriented so that $\gamma_k \cdot \beta_k = 1$,
			\item $D^2 \subset T^2$ does not lie in one of the two triangular regions of $T^2\setminus (\beta_k \cup \gamma_k \cup \delta_k)$.
		\end{itemize}
	\end{itemize}See \cite[Section 4]{MR2601010} for a proof of existence of such a configuration. 
	The condition that $D^2 \subset T^2$ does not lie in a triangular region is for admissibility of $(\Sigma,\bo\beta,\bo\gamma,\bo\delta)$. 
	Choose a point $p \in \beta_k$ that does not lie on any of the other curves, and let $\omega \subset \Sigma$ be an isotopic translate of $\gamma_k$ that intersects $\beta_k$ at $p$. Note that $\omega$ can be oriented so that with respect to the diagram $(\Sigma,\bo\alpha,\bo\beta)$ of $(M,\gamma)$, it represents the class $\zeta = [K] \in H_1(M,\partial M)$. For any Whitney disc $\phi$ of $(\Sigma,\bo\alpha,\bo\beta)$, observe that $\omega\cdot\partial_\beta D(\phi)$ is just $m_p(\phi)$, following the notation of \cite[Theorem 3.1]{MR2377279}. 

	There are canonical generators $\Theta_{\gamma\delta}, \Theta_{\delta\beta},\Theta_{\beta\gamma}$ lying in $\mathbf{T}_\gamma \cap \mathbf{T}_\delta, \mathbf{T}_\delta \cap \mathbf{T}_\beta, \mathbf{T}_\beta \cap \mathbf{T}_\gamma$, respectively. For each $i = 1,\ldots,k-1$, the curves $\gamma_i$ and $\delta_i$ intersect in two points $x_i,y_i$, labeled so that there are two bigon domains from $x_i$ to $y_i$ in $(\Sigma,\bo\gamma,\bo\delta)$. The canonical generator $\Theta_{\gamma\delta}$ contains $x_i$ for $i = 1,\ldots,k-1$. The analogous description holds for $\Theta_{\delta\beta}$ and $\Theta_{\beta\gamma}$. Since $\delta_k \cap \beta_k$ and $\beta_k \cap \gamma_k$ each consist of a single point, the generators $\Theta_{\delta\beta}$ and $\Theta_{\beta\gamma}$ are now determined. The generator $\Theta_{\gamma\delta}$ contains one of the $n$ points of $\gamma_k \cap \delta_k$ which can be specified by the relative $\Spin^c$ in which the generator lies (see \cite[Definition 3.2]{MR2377279}). 

	The rest of the proof is the same as the proof of \cite[Theorem 3.1]{MR2377279}. There are maps\[
		\begin{tikzcd}[column sep=-30pt,row sep=50pt]
			\SFC(\Sigma,\bo\alpha,\bo\gamma) \ar[rr,"f_1"] \ar[dr,swap,bend right=40pt,"H_1"] & & \SFC(\Sigma,\bo\alpha,\bo\delta) \ar[ld,"f_2"] \ar[ll,swap,bend right,"H_2"]\\
			& \SFC(\Sigma,\bo\alpha,\bo\beta;\F[\Z/n]_\omega) \ar[ul,"f_3"] \ar[ur,swap,bend right=40pt,"H_3"] &
		\end{tikzcd}
	\]where $H_i$ is a nullhomotopy for $f_{i + 1} \circ f_{i}$ for $i \in \Z/3$. The maps $f_i$ are defined by counts of pseudo-holomorphic triangles, and the maps $H_i$ are defined by counts of pseudo-holomorphic rectangles. The counts are made with suitable bookkeeping of the multiplicities of the domains at $p$. The triangle detection lemma is then applied, by choosing suitable translates $\bo\beta',\bo\gamma',\bo\delta'$ and counting pseudo-holomorphic rectangles and pentagons. 
\end{proof}

\subsection{Band surgeries}\label{subsec:bandSurgeries}

Let $L$ be an oriented link in $S^3$ with basepoints $q,r$ on distinct components. Let $b$ be an orientation-preserving band that merges the marked components of $L$, and let $K_b$ be the oriented link obtained from $L$ by band surgery along $b$. Let $K_{b+n}$ be obtained by adding $n$ full twists to the band. Let $C$ be the knot in the complement of $K_b$ that bounds a disc $D$ that intersects $b$ along a cocore and is otherwise disjoint from $K_b$. Our convention for a full twist is that $(-1/n)$-surgery along $C$ takes $K_b$ to $K_{b+n}$. 

Let $\kappa(L,q,r) \subset S^1 \x S^2$ be the link obtained from $K_b$ by doing $0$-surgery along $C$, and let $C'$ be the core of the surgery. Observe that the link $\kappa(L,q,r)$ is independent of the band $b$, and that the homology class of $C'$ in the exterior of $\kappa(L,q,r)$ is also independent of the band $b$. We note that $\kappa(L,q,r)$ is obtained from $L$ with its basepoints $q,r$ from a version of the ``knotification'' construction in \cite[Section 2]{MR2065507}. 

The dimension of the knot Floer homology $\HFKhat(K_{b+n};\Q)$ of the link $K_{b+n}$ may depend on the choice of system of orientations (see \cite{MR2812456}). However, there is a \textit{canonical} choice \cite[section 3]{MR3604486} which we use by default for all knot Floer homology groups in this paper because Dowlin's spectral sequence \cite{dowlin2018spectral} is defined with respect to this canonical choice. 

\begin{lem}\label{lem:skeinTriangleTwistedCoeff}
	For each $n \ge 1$, there is an exact triangle \[
		\begin{tikzcd}[column sep = -5em, row sep = large]
			\HFKhat(K_b;\Q) \ar[rr] & & \HFKhat(K_{b-n};\Q) \ar[dl]\\
			& \underline{\SFH}((S^1\x S^2)(\kappa(L,q,r));\Q[\Z/n]_{[C']}) \ar[ul] & 
		\end{tikzcd}
	\]where $(S^1 \x S^2)(\kappa(L,q,r))$ is the sutured exterior of $\kappa(L,q,r) \subset S^1 \x S^2$ equipped with any system of orientations. 
\end{lem}
\begin{proof}
	This is the exact triangle of Proposition~\ref{prop:surgeryTriangleTwisted} in the case that the sutured manifold is $(S^1 \x S^2)(\kappa(L,q,r))$ and the framed knot is $C'$ whose framing is given by the meridian of $C$. Integral surgeries along $C'$ are the $(1/n)$-surgeries along $C$, and our convention for the definition of a full twist is that $1/n$ surgery on $C$ takes $K_b$ to $K_{b-n}$. Finally, there is an identification $\HFKhat(K_b;\Q) = \SFH(S^3(K_b);\Q)$ given by \cite[Proposition 9.2]{Juh06}.
\end{proof}

\begin{proof}[Proof of Theorem~\ref{thm:unboundedHFKhat}]
	It suffices to consider the case $n \to -\infty$ because $\dim_\Q \HFKhat$ is invariant under mirroring. By exactness of the triangle of Lemma~\ref{lem:skeinTriangleTwistedCoeff}, it suffices to show that \[
		\dim_\Q \underline{\SFH}((S^1 \x S^2)(\kappa(L,q,r));\Q[\Z/n]_{[C']}) \to \infty \quad\text{ as }\quad n \to \infty.
	\]This criterion is independent of $b$. Observe that there is a nice surface decomposition \[
		(S^1 \x S^2)(\kappa(L,q,r)) \overset{A}{\rightsquigarrow} S^3(L)
	\]where $S^3(L)$ is the sutured exterior of $L \subset S^3$, for which $\zeta = i_*([C']) \in H_1(S^3(L),\partial S^3(L))$ is the relative homology class of the core of the band $b$. The surface $A$ is a product annulus, obtained in the following way. The disc $D$ that $C$ bounds in $S^3$ that intersects the band $b$ along a cocore may be viewed as a pair of pants in the exterior of $K_b \cup C$. After Dehn filling the toral boundary component corresponding to $C$, we cap off the pair of pants with a disc to obtain the annulus $A$. By Proposition~\ref{prop:twistedDirectSummandSurfDecomp}, it suffices to show that \[
		\dim_\Q \underline{\SFH}(S^3(L);\Q[\Z/n]_{\zeta}) \to \infty \quad\text{ as }\quad n \to \infty.
	\]We will show that $\dim_\Q \underline{\SFH}(S^3(L);\Q[\Z/n]_\zeta) \ge n$. 

	If every embedded $2$-sphere in the complement of $L$ bounds a ball, then $S^3(L)$ is a taut sutured manifold so $\dim_\Q \underline{\SFH}(S^3(L);\Q[\Z/n]_\zeta) \ge n$ by Corollary~\ref{cor:dimInequalTaut}. Otherwise, we use a split union formula to reduce to this case. Write $L$ as the split union of two nonempty links $L' \amalg L''$. Since no sphere in $S^3\setminus L$ separates $q$ and $r$, we may assume that both basepoints lie on $L'$. We may also assume that $S^3(L')$ is irreducible. By \cite[Proposition 9.15]{Juh06}, there is a product disc decomposition \[
		S^3(L) \overset{D}{\rightsquigarrow} S^3(L') \amalg S^3(L'')(1)
	\]so by Proposition~\ref{prop:productDiscDecomp}, we have an isomorphism \[
		\underline{\SFH}(S^3(L);\Q[\Z/n]_\zeta) \cong \underline{\SFH}(S^3(L') \amalg S^3(L'')(1);\Q[\Z/n]_{i_*(\zeta)}).
	\]Observe that \[
		i_*(\zeta) = \zeta' \oplus 0 \in H_1(S^3(L'),\partial S^3(L')) \oplus H_1(S^3(L'')(1),\partial S^3(L'')(1))
	\]where $\zeta'$ is the relative homology class of an arc whose basepoints lie on the toral boundary components of $S^3(L')$ marked by $p$ and $q$. It then follows from the definition of the sutured Floer chain complex with twisted coefficients and a suitably chosen balanced diagram that \[
		\underline{\SFH}(S^3(L') \amalg S^3(L'')(1);\Q[\Z/n]_{i_*(\zeta)}) \cong \underline{\SFH}(S^3(L');\Q[\Z/n]_{\zeta'}) \otimes_\Q \SFH(S^3(L'')(1);\Q).
	\]Since $S^3(L')$ is taut, it follows from Corollary~\ref{cor:dimInequalTaut} that $\dim_\Q \underline{\SFH}(S^3(L);\Q[\Z/n]_\zeta) \ge n$. 
\end{proof}

\section{$\sl(N)$ link homology}\label{sec:KRsplitLinkDetect}

In section~\ref{subsec:slNprelims}, we briefly describe the version of $\sl(N)$ link homology that we use. Its construction uses a cube of resolutions and Robert--Wagner's combinatorial evaluation of closed foams \cite{MR4164001}. We refer to \cite{wang2021sln} for a detailed exposition of the construction, and for the definitions of $\sl(N)$ MOY graphs and $\sl(N)$ foams. In section~\ref{subsec:slNsplitLinkDetection}, we prove the results stated in the introduction. 

\subsection{Preliminaries}\label{subsec:slNprelims}

Let $D$ be an oriented link diagram. If $c(D)$ denotes the set of crossings of $D$, then to each function $v\colon c(D) \to \{0,1\}$, we associate an $\sl(N)$ MOY graph $D_v$ according to Figure~\ref{fig:resolutions}. Each edge of $D_v$ carries an orientation and a label of either $1$ or $2$. If $v,w\colon c(D)\to\{0,1\}$ agree at all crossings except at one where $v(c) = 0$ and $w(c) = 1$, then there is an $\sl(N)$ foam $F_{vw}\colon D_v \to D_w$.

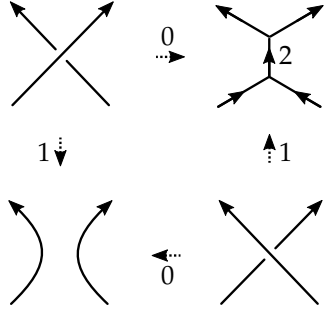
\begin{figure}[!ht]
	\labellist
	\pinlabel $1$ at 65 280
	\pinlabel $1$ at 505 280
	\pinlabel $0$ at 290 490
	\pinlabel $0$ at 290 55
	\pinlabel $2$ at 505 455
	\endlabellist
	\centering
	\includegraphics[width=.3\textwidth]{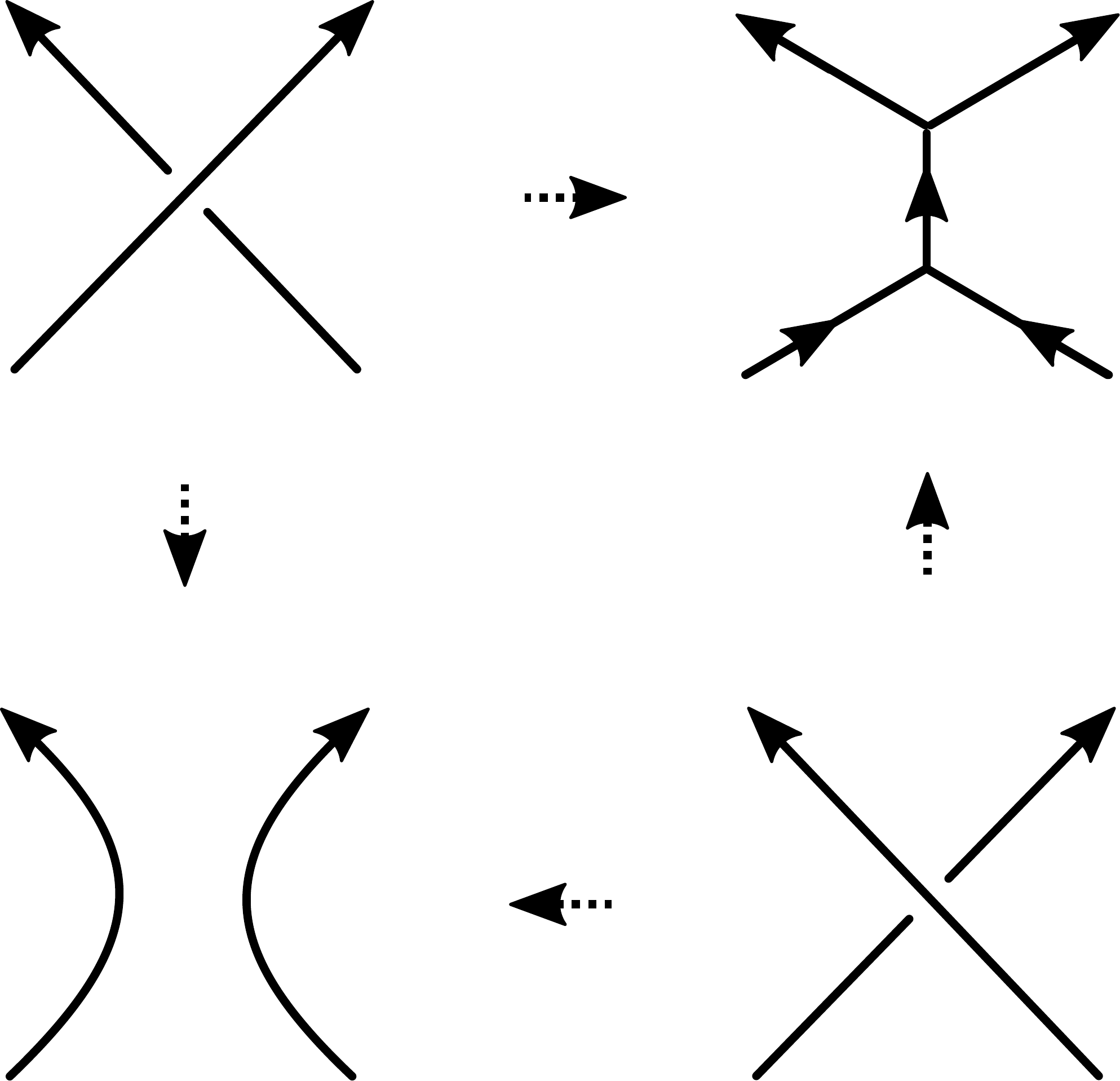}
	\captionsetup{width=.8\linewidth}
	\caption{MOY graphs obtained by resolving crossings of an oriented link diagram. An edge of an MOY graph without an explicit label is labeled $1$. The top left crossing is positive and the bottom right crossing is negative.}
	\label{fig:resolutions}
\end{figure}

Let $R$ be a ring. Associated to each $D_v$ is a free $R$-module $\sr F_N(D_v;R)$, and associated to each $F_{vw}\colon D_v\to D_w$ is an $R$-module map $\sr F_N(F_{vw};R) \colon \sr F_N(D_v;R) \to \sr F_N(D_w;R)$. The $\sl(N)$ chain complex $\KRC_N(D;R)$ is the direct sum of the $R$-modules $\sr F_N(D_v;R)$ with differential given by a signed sum of the maps $\sr F_N(F_{vw};R)$. The (unreduced) $\sl(N)$ link homology $\KR_N(L;R)$ of the link $L$ that $D$ represents is the homology of the chain complex $\KRC_N(D;R)$. There is an identification $\KRC_N(D;R) = \KRC_N(D;\Z) \otimes R$ so $\KR_N(L;R)$ satisfies the universal coefficient theorem. When $N = 2$, there is an identification $\KR_2(L;R) \cong \Kh(m(L);R)$ where $m(L)$ denotes the mirror of $L$. 

A basepoint $q \in D$ away from the crossings can be viewed as a basepoint on each $D_v$. The basepoint $q \in D_v$ induces a basepoint operator $X_q\colon \sr F_N(D_v;R) \to \sr F_N(D_v;R)$ satisfying $X_q^N = 0$ (see for example \cite[section 2.6]{wang2021sln}). We define $X_q$ on $\KRC_N(D;R)$ to be the sum of these maps defined on the individual $\sr F_N(D_v;R)$. This operator $X_q$ on $\KRC_N(D;R)$ is a chain map. The reduced $\sl(N)$ link homology $\rKR_N(L,q;R)$ of $L$ with respect to $q$ is defined to be the homology of the subcomplex $X_q^{N-1}\KRC_N(D;R)$. If $r \in D$ is an additional basepoint, then the associated operator $X_r\colon \KRC_N(D;R) \to \KRC_N(D;R)$ commutes with $X_q$ and induces a basepoint operator on $\rKR_N(L,q;R)$ satisfying $X_r^N = 0$. 

\vspace{5pt}

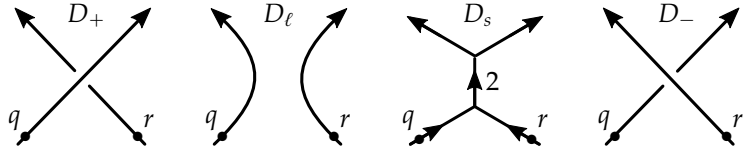
\begin{figure}[!ht]
	\centering
	\labellist
	\pinlabel $q$ at 0 35
	\pinlabel $\bullet$ at 15 12
	\pinlabel $q$ at 272 35
	\pinlabel $\bullet$ at 287 12
	\pinlabel $q$ at 544 35
	\pinlabel $\bullet$ at 559 8
	\pinlabel $q$ at 814 35
	\pinlabel $\bullet$ at 829 12
	\pinlabel $r$ at 185 35
	\pinlabel $\bullet$ at 170 12
	\pinlabel $r$ at 456 35
	\pinlabel $\bullet$ at 441 12
	\pinlabel $r$ at 729 35
	\pinlabel $\bullet$ at 714 8
	\pinlabel $r$ at 998 35
	\pinlabel $\bullet$ at 983 12
	\pinlabel $2$ at 660 90
	\pinlabel $D_+$ at 100 180
	\pinlabel $D_\ell$ at 365 180
	\pinlabel $D_s$ at 640 180
	\pinlabel $D_-$ at 915 180
	\endlabellist
	\includegraphics[width=.13\textwidth]{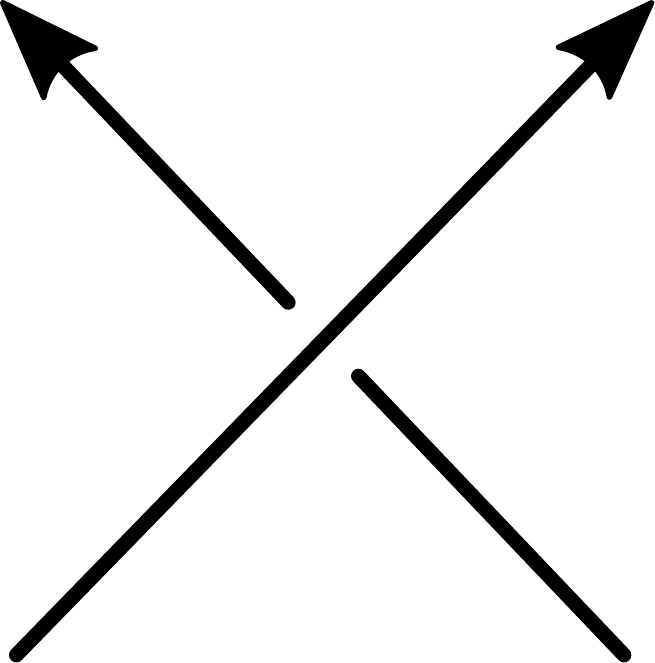}\hspace{20pt}
	\includegraphics[width=.13\textwidth]{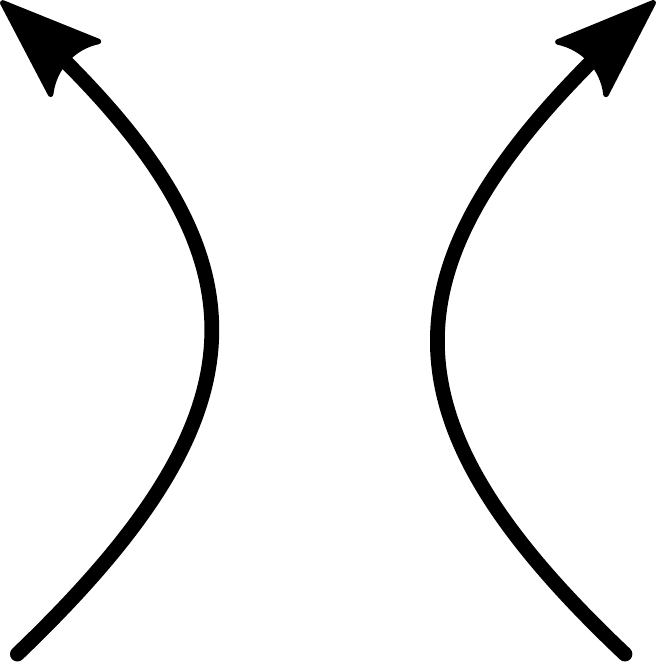}\hspace{20pt}
	\includegraphics[width=.13\textwidth]{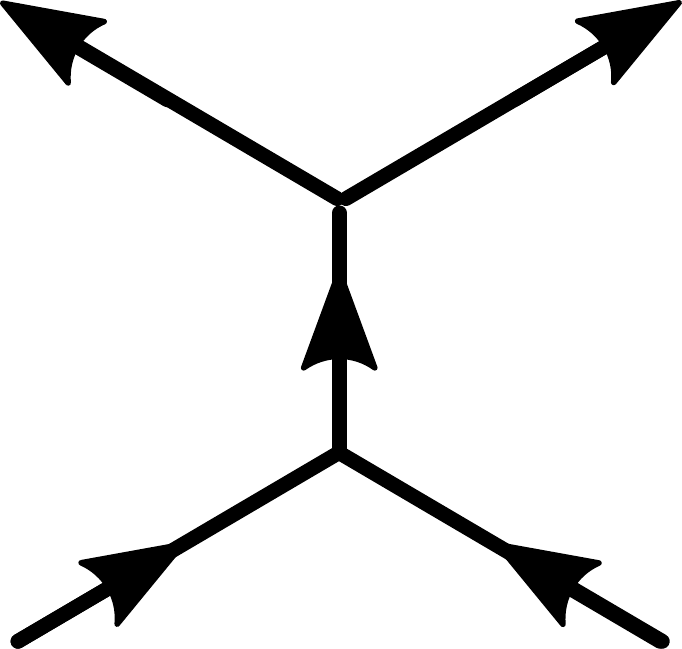}\hspace{20pt}
	\includegraphics[width=.13\textwidth]{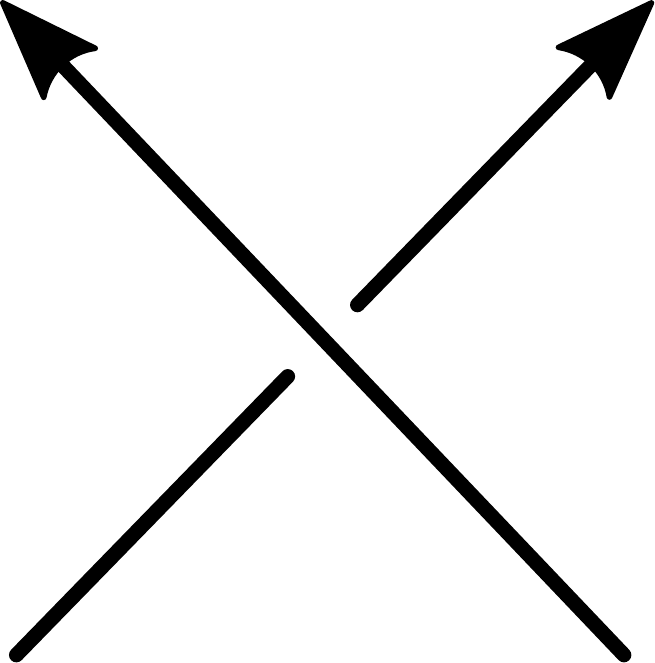}
	\caption{Local diagrams with basepoints for $D_+,D_\ell,D_s,D_-$}
	\label{fig:basepointsOnResolutions}
\end{figure}

Let $D_+$ be a diagram with a fixed positive crossing $c$ such that the two strands at $c$ lie on the same component of the link that $D_+$ represents. Let $D_\ell, D_s, D_-$ be the diagrams obtained by modifying $D_+$ near $c$ according to Figure~\ref{fig:basepointsOnResolutions}. Note that $D_\ell$ represents a link having one more component than the link that $D_+$ represents. We fix basepoints $q,r$ on each of the four diagrams as shown in the figure. 

We associate a chain complex $\KRC_N(D_s;R)$ to $D_s$ in the natural way by using a cube of resolutions. In particular, $\KRC_N(D_s;R)$ is a subcomplex of $\KRC_N(D_-;R)$ and a quotient complex of $\KRC_N(D_+;R)$. Just like the chain complex associated to an ordinary link diagram, a basepoint $q \in D_s$ away from the crossings and away from the edge labeled $2$ induces a chain map $X_q$ on $\KRC_N(D_s;R)$ satisfying $X_q^N = 0$. The edge labeled $2$, however, gives rise to two different basepoint operators on $\KRC_N(D_s;R)$: a weight 1 operator $E_1$ and a weight 2 operator $E_2$ (see for example \cite[section 2.6]{wang2021sln}). These operators commute and satisfy certain universal relations that we now explain. 

For a basepoint operator $X_q$ associated to an edge labeled $1$, the identity $\smash{X_q^N} = 0$ can be viewed as the assertion that the action of the polynomial ring $R[X]$ on $\KRC_N(D_s;R)$ given by $X = X_q$ descends to an action of $H^*(\CP^{N-1};R) = R[X]/X^N$ on $\KRC_N(D_s;R)$ (see for example \cite[Lemma 2.17]{wang2021sln}). Similarly, there is an action of the ring of symmetric polynomials $R[X_1,X_2]^{\fk S_2}$ on $\KRC_N(D_s;R)$ where the first elementary symmetric polynomial $e_1 = X_1 + X_2$ acts on $\KRC_N(D_s;R)$ by $E_1$ and the second elementary symmetric polynomial $e_2 = X_1X_2$ acts on $\KRC_N(D_s;R)$ by $E_2$. The cohomology of the Grassmannian $\G(2,N)$ of $2$-planes in $\C^N$ can be viewed as a quotient of $R[X_1,X_2]^{\fk S_2}$, where $e_1$ corresponds to a degree $2$ cohomology class and $e_2$ corresponds to a degree $4$ cohomology class. An argument similar to \cite[Lemma 2.17]{wang2021sln} shows that the action of $R[X_1,Y_2]^{\fk S_2}$ on $\KRC_N(D_s;R)$ descends to an action of $H^*(\G(2,N);R)$ on $\KRC_N(D_s;R)$ (see for example \cite[section 5]{wang2021sln}). Since $\G(2,N)$ has real dimension $4(N-2)$, it follows that $e_2^{N-1} = 0$ in $H^*(\G(2,N);R)$ so $E_2^{N-1} = 0$. We use this identity below. 

\vspace{10pt}

By construction, there are short exact sequences of chain complexes\[
	\begin{tikzcd}[row sep = 0, column sep=2em]
		0 \ar[r] & X_q^{N-1}\KRC_N(D_\ell;R) \ar[r] & X_q^{N-1}\KRC_N(D_+;R) \ar[r] & X_q^{N-1}\KRC_N(D_s;R) \ar[r] & 0\\
		0 \ar[r] & X_q^{N-1}\KRC_N(D_s;R) \ar[r] & X_q^{N-1}\KRC_N(D_-;R) \ar[r] & X_q^{N-1}\KRC_N(D_\ell;R) \ar[r] & 0
	\end{tikzcd}
\]that induce exact triangles \[
	\begin{tikzcd}[column sep=-2em]
		\rKR_N(D_+,q;R) \ar[rr] & & \rKR_N(D_s,q;R) \ar[ld]\\
		& \rKR_N(D_\ell,q;R) \ar[lu] &
	\end{tikzcd}\qquad \begin{tikzcd}[column sep=-2em]
		\rKR_N(D_s,q;R) \ar[rr] & & \rKR_N(D_-,q;R) \ar[ld]\\
		& \rKR_N(D_\ell,q;R) \ar[lu] &
	\end{tikzcd}
\]where $\rKR_N(D_s,q;R)$ is the homology of $X_q^{N-1}\KRC_N(D_s;R)$. These exact triangles are the direct generalization of the skein exact triangles in Khovanov homology. There is a basepoint operator $X_r$ defined on each of these reduced homology groups, and the maps in the exact triangles intertwine them. Since the basepoint operator $X_r$, at the level of homology, depends only on the component of the link carrying the basepoint (see for example \cite[Lemma 5.16]{MR3447099}), it follows that $X_r = X_q = 0$ on $\rKR_N(D_+,q;R)$ and $\rKR_N(D_-,q;R)$. 

\begin{lem}\label{lem:Nminus1powerZeroOnDs}
	On $\rKR_N(D_s,q;R)$, the basepoint operator $X_r$ satisfies $X_r^{N-1} = 0$. 
\end{lem}
\begin{proof}
	We prove the result at the chain level. Let $E_1$ and $E_2$ be the weight 1 and weight 2 basepoint operators on $\KRC_N(D_s;R)$ associated to the edge labeled $2$. The dot-migration relation \cite[Proposition 3.32 (11)]{MR4164001} implies that $X_rX_q = E_2$. Since $E_2^{N-1} = 0$, it follows that $X_r^{N-1} = 0$ on $X_q^{N-1}\KRC_N(D_s;R)$.
\end{proof}

\subsection{Proofs of the main results}\label{subsec:slNsplitLinkDetection}

\begin{proof}[Proof of Proposition~\ref{prop:freeImpliesDimIndependent}]
	Suppose $\rKR_N(L,q;\F)$ is a free module over $\F[X]/X^N$ where $X = X_r$, and let $b$ be a band that merges the marked components of $L$. There is a diagram $D_\ell$ of $L$ such that $K_b$ is given by $D_-$ and $K_{b+1}$ is given by $D_+$, where $D_-$, $D_+$, and $D_s$ are obtained from $D_\ell$ according to Figure~\ref{fig:basepointsOnResolutions}. Furthermore, we may assume that the basepoints $q,r$ are as shown in the figure. Consider the exact triangle \[
		\begin{tikzcd}[column sep=-1em,row sep=3em]
			\rKR_N(K_{b+1},q;\F) \ar[loop above,dotted,"X_r = 0"] \ar[rr,"h"] & & \rKR_N(D_s,q;\F) \ar[loop above,dotted,"X_r = Y"]\ar[ld,"f"]\\
			& \rKR_N(L,q;\F). \ar[lu,"g"] \ar[loop below,dotted,"X_r = X"] &
		\end{tikzcd}
	\]where the solid arrows intertwine the dotted arrows. For notational convenience, we let $X$ denote the operator $X_r$ on $\rKR_N(L,q;\F)$, and we let $Y$ denote $X_r$ on $\rKR_N(D_s,q;\F)$. By Lemma~\ref{lem:Nminus1powerZeroOnDs}, we have $Y^{N-1} = 0$. Because $\rKR_N(L,q;\F)$ is free over $\F[X]/X^N$, we know that $\Image(X) = \ker(X^{N-1})$. Now observe that \[
		\Image(f) \subseteq \ker(X^{N-1}) = \Image(X) \subseteq \ker(g)
	\]where the first containment follows from the identities $Y^{N-1} = 0$ and $f\circ Y = X\circ f$. The second containment follows by similar reasoning. 
	By exactness of the triangle, the containments are equalities. Write $\dim \rKR_N(L,q;\F) = N\cdot d$, so that \[
		\rk(g) = \dim \rKR_N(L,q;\F) - \dim \ker(g) = N d - (N-1) d = d.
	\]It follows that $\dim \rKR_N(K_{b+1},q;\F) = d + \rk(h)$ and $\dim \rKR_N(D_s,q;\F) = (N-1)d + \rk(h)$. Similar reasoning applied to the exact triangle \[
		\begin{tikzcd}[column sep=-1em,row sep=3em]
			\rKR_N(D_s,q;R) \ar[rr,"c"] & & \rKR_N(K_b,q;R) \ar[ld,"a"]\\
			& \rKR_N(L,q;R) \ar[lu,"b"] &
		\end{tikzcd}
	\]yields $\dim \rKR_N(K_b,q;\F) = d + \rk(c)$ and $\dim \rKR_N(D_s,q;\F) = (N-1)d + \rk(c)$. It follows that $\rk(c) = \rk(h)$ so \[
		\dim \rKR_N(K_{b+1},q;\F) = \dim \rKR_N(K_b,q;\F).\qedhere
	\]
\end{proof}

\begin{proof}[Proof of Proposition~\ref{prop:threeEquivalentStatements}]
	It is clear that the second statement implies the third. We show that the third implies the first. Assume that we have a band $b$ for which $\dim \rKR_N(K_{b+n},q;\F) < M$ for some $M$ independent of $n$. By the universal coefficient theorem, we know that \[
		\dim \rKR_N(K_{b+n},q;\Q) \leq \dim \rKR_N(K_{b+n},q;\F) < M.
	\]Next, we show that $\dim \rKh(m(K_{b+n}),q;\Q) < M$ where $m(K_{b+n})$ denotes the mirror of $K_{b+n}$. This is certainly true if $N = 2$ because $\rKh(m(K_{b+n});\Q) \cong \rKR_2(K_{b+n};\Q)$. For $N > 2$, the rank-reducing spectral sequence \cite[Theorem 5]{MR3982970} yields \[
		\dim \rKR_{N-1}(K_{b+n},q;\Q) \leq \dim \rKR_N(K_{b+n},q;\Q)
	\]from which the claim follows by induction. \cite[Theorem 5]{MR3982970} is stated for knots and the corresponding result for links is left to the reader. The relevant spectral sequence for our purposes has $E_2$-page $\rKR_N(K_{b+n},q;\Q)$ and converges to the direct sum of the vector spaces $\rKR_{N-1}(J,q;\Q)$ over all sublinks $J \subseteq K_{b+n}$ that contain the component of $K_{b+n}$ marked by $q$. The direct summand corresponding to $J = K_{b+n}$ gives the stated inequality. 
	Now by \cite[Corollary 1.7]{dowlin2018spectral}, we have the inequality \[
		\dim \HFKhat(K_{b+n};\Q) \leq 2^{|L| - 2}\rKh(m(K_{b+n}),q;\Q) < 2^{|L| - 2}M
	\]By Theorem~\ref{thm:unboundedHFKhat}, there is an embedded sphere in $S^3\setminus L$ that separates $q$ and $r$. 

	The first statement implies the second by the argument in \cite{wang2020cosmetic} used to the prove the analogous statement for Khovanov homology, knot Floer homology, and instanton knot homology. We give a detailed sketch of the proof. Assume that there is a sphere in $S^3\setminus L$ separating $q$ and $r$, and let $b$ be a band merging the marked components of $L$. Just as in the proof of Proposition~\ref{prop:freeImpliesDimIndependent}, there is a diagram $D_\ell$ of $L$ with associated diagrams $D_-,D_+,D_s$ according to Figure~\ref{fig:basepointsOnResolutions} for which $D_-$ represents $K_b$ and $D_+$ represents $K_{b+1}$. Consider the associated exact triangles \[
		\begin{tikzcd}[column sep=-2em]
			\rKR_N(K_{b+1},q;\F) \ar[rr] & & \rKR_N(D_s,q;\F) \ar[ld]\\
			& \rKR_N(L,q;\F) \ar[lu,"g"] &
		\end{tikzcd}\qquad \begin{tikzcd}[column sep=-2em]
			\rKR_N(D_s,q;\F) \ar[rr] & & \rKR_N(K_b,q;\F) \ar[ld,"a"]\\
			& \rKR_N(L,q;\F) \ar[lu] &
		\end{tikzcd}
	\]It is straightforward to verify that $\dim \rKR_N(K_{b+1},q;\F) = \dim \rKR_N(K_b,q;\F)$ if and only if $\rk(g) = \rk(a)$.

	Assume that $L$ is the split union of links $L'$ and $L''$ where $q \in L'$ and $r \in L''$. Let $K_\#$ denote the link obtained by forming the connected sum of $L'$ with $L''$ along the components marked by $q$ and $r$. Note that $K_\#$ is the special case of $K_b$ when $b$ is trivial. We may choose diagrams $D_\ell^\#,D_-^\#,D_+^\#,D_s^\#$ so that $D_\ell^\#$ represents $L$, and $D_-^\#$ and $D_+^\#$ both represent $K_\#$. There are associated exact triangles \[
		\begin{tikzcd}[column sep=-2em]
			\rKR_N(K_\#,q;\F) \ar[rr] & & \rKR_N(D_s^\#,q;\F) \ar[ld]\\
			& \rKR_N(L,q;\F) \ar[lu,"g_\#"] &
		\end{tikzcd}\qquad \begin{tikzcd}[column sep=-2em]
			\rKR_N(D_s^\#,q;\F) \ar[rr] & & \rKR_N(K_\#,q;\F) \ar[ld,"a_\#"]\\
			& \rKR_N(L,q;\F) \ar[lu] &
		\end{tikzcd}
	\]from which it follows that $\rk(g_\#) = \rk(a_\#)$. We now show that $\rk(g) = \rk(g_\#)$ and $\rk(a) = \rk(a_\#)$ to obtain the equality $\rk(g) = \rk(a)$. 

	By \cite{MR1451821}, there is a ribbon concordance $C_b$ from $K_\#$ to $K_b$. There is an induced map \[
		\rKR_N(C_b;\F)\colon \rKR_N(K_\#;\F) \to \rKR_N(K_b;\F)
	\]which can be shown to be injective \cite{kang2019link,caprau2020khovanov} ultimately based on an argument of Zemke for knot Floer homology \cite{Zem19a}. By the proof of \cite[Proposition 5.7]{wang2020cosmetic}, there are injective maps (displayed below as dotted) making the diagram \[
		\begin{tikzcd}[column sep=small]
			\rKR_N(D_s,q;\F) \ar[rr] & & \rKR_N(K_b,q;\F) \ar[dl,"a"]\\
			& \rKR_N(L,q;\F) \ar[lu] & \\
			\rKR_N(D_s^\#,q;\F) \ar[uu,hook,dotted] \ar[rr] & & \rKR_N(K_\#,q;\F) \ar[dl,"a_\#"] \ar[uu,hook,swap,"\rKR_N(C_b;\F)"]\\
			& \rKR_N(L,q;\F) \ar[uu,hook,dotted,crossing over] \ar[ul] &
		\end{tikzcd}
	\]commute. Because $\rKR_N(L,q;\F)$ is finite-dimensional, the injective map from $\rKR_N(L,q;\F)$ to itself is an isomorphism. It follows that $\rk(a) = \rk(a_\#)$. The proof that $\rk(g) = \rk(g_\#)$ is similar. 
\end{proof}

\begin{proof}[Proof of Theorem~\ref{thm:freeImpliesSplit}]
	The result follows from Propositions~\ref{prop:freeImpliesDimIndependent} and \ref{prop:threeEquivalentStatements}. 
\end{proof}
\begin{proof}[Proof of Theorem~\ref{thm:mainthm}]
	If $\rKR_P(L,q;\F_P)$ is free over $\F_P[X]/X^P$, then there is an embedded sphere in $S^3\setminus L$ separating $q$ and $r$ by Theorem~\ref{thm:freeImpliesSplit}. If $L$ is the split union of links $L'$ and $L''$ where $q \in L'$ and $r \in L''$, then there is an isomorphism\[
		\rKR_P(L,q;\F_P) \cong \rKR_P(L',q;\F_P) \otimes \KR_P(L'';\F_P)
	\]which intertwines $X_r$ on the left with $\Id \otimes\, X_r$ on the right by \cite[Corollary 2.20]{wang2021sln}. It suffices to show that $\KR_P(L'';\F_P)$ is free over $\F_P[X]/X^P$ where $X = X_r$, which follows from \cite[Theorem 1.1]{wang2021sln}. 
\end{proof}

\raggedright
\bibliography{KRsplitLinks}

\newcommand{\etalchar}[1]{$^{#1}$}
\begin{thebibliography}{CGL{\etalchar{+}}20}

\bibitem[AE15]{MR3412088}
Akram~S. Alishahi and Eaman Eftekhary.
\newblock A refinement of sutured {F}loer homology.
\newblock {\em J. Symplectic Geom.}, 13(3):609--743, 2015.

\bibitem[BDL{\etalchar{+}}21]{MR4275096}
John~A. Baldwin, Nathan Dowlin, Adam~Simon Levine, Tye Lidman, and Radmila
  Sazdanovic.
\newblock Khovanov homology detects the figure-eight knot.
\newblock {\em Bull. Lond. Math. Soc.}, 53(3):871--876, 2021.

\bibitem[BHS21]{baldwin2021khovanovcinquefoil}
John~A. Baldwin, Ying Hu, and Steven Sivek.
\newblock Khovanov homology and the cinquefoil.
\newblock arXiv:2105.12102, 2021.

\bibitem[BLS17]{MR3604486}
John~A. Baldwin, Adam~Simon Levine, and Sucharit Sarkar.
\newblock Khovanov homology and knot {F}loer homology for pointed links.
\newblock {\em J. Knot Theory Ramifications}, 26(2):1740004, 49, 2017.

\bibitem[BS15]{MR3332892}
Joshua Batson and Cotton Seed.
\newblock A link-splitting spectral sequence in {K}hovanov homology.
\newblock {\em Duke Math. J.}, 164(5):801--841, 2015.

\bibitem[BS21]{baldwin2021khovanovtrefoil}
John~A. Baldwin and Steven Sivek.
\newblock Khovanov homology detects the trefoils.
\newblock arXiv:1801.07634, 2021.

\bibitem[BSX19]{MR4049809}
John~A. Baldwin, Steven Sivek, and Yi~Xie.
\newblock Khovanov homology detects the {H}opf links.
\newblock {\em Math. Res. Lett.}, 26(5):1281--1290, 2019.

\bibitem[CGL{\etalchar{+}}20]{caprau2020khovanov}
Carmen Caprau, Nicolle González, Christine Ruey~Shan Lee, Adam~M. Lowrance,
  Radmila Sazdanović, and Melissa Zhang.
\newblock On {K}hovanov homology and related invariants.
\newblock arXiv:2002.05247, 2020.

\bibitem[Dow18]{dowlin2018spectral}
Nathan Dowlin.
\newblock A spectral sequence from {K}hovanov homology to knot {F}loer
  homology.
\newblock arXiv:1811.07848, 2018.

\bibitem[Gab83]{Gab83}
David Gabai.
\newblock Foliations and the topology of {$3$}-manifolds.
\newblock {\em J. Differential Geom.}, 18(3):445--503, 1983.

\bibitem[Gab87]{Gab87a}
David Gabai.
\newblock Foliations and the topology of {$3$}-manifolds. {II}.
\newblock {\em J. Differential Geom.}, 26(3):461--478, 1987.

\bibitem[GW10]{MR2601010}
J.~Elisenda Grigsby and Stephan~M. Wehrli.
\newblock On the colored {J}ones polynomial, sutured {F}loer homology, and knot
  {F}loer homology.
\newblock {\em Adv. Math.}, 223(6):2114--2165, 2010.

\bibitem[HN10]{MR2653731}
Matthew Hedden and Yi~Ni.
\newblock Manifolds with small {H}eegaard {F}loer ranks.
\newblock {\em Geom. Topol.}, 14(3):1479--1501, 2010.

\bibitem[HN13]{MR3190305}
Matthew Hedden and Yi~Ni.
\newblock Khovanov module and the detection of unlinks.
\newblock {\em Geom. Topol.}, 17(5):3027--3076, 2013.

\bibitem[JMZ20]{juhasz2020transverse}
Andr\'{a}s Juh\'{a}sz, Maggie Miller, and Ian Zemke.
\newblock Transverse invariants and exotic surfaces in the 4-ball.
\newblock arXiv:2001.07191, 2020.

\bibitem[Juh06]{Juh06}
Andr\'{a}s Juh\'{a}sz.
\newblock Holomorphic discs and sutured manifolds.
\newblock {\em Algebr. Geom. Topol.}, 6:1429--1457, 2006.

\bibitem[Juh08]{Juh08}
Andr\'{a}s Juh\'{a}sz.
\newblock Floer homology and surface decompositions.
\newblock {\em Geom. Topol.}, 12(1):299--350, 2008.

\bibitem[Juh10]{MR2653728}
Andr\'{a}s Juh\'{a}sz.
\newblock The sutured {F}loer homology polytope.
\newblock {\em Geom. Topol.}, 14(3):1303--1354, 2010.

\bibitem[Kan19]{kang2019link}
Sungkyung Kang.
\newblock Link homology theories and ribbon concordances.
\newblock arXiv:1909.06969, 2019.

\bibitem[Kho00]{MR1740682}
Mikhail Khovanov.
\newblock A categorification of the {J}ones polynomial.
\newblock {\em Duke Math. J.}, 101(3):359--426, 2000.

\bibitem[KM11]{MR2805599}
P.~B. Kronheimer and T.~S. Mrowka.
\newblock Khovanov homology is an unknot-detector.
\newblock {\em Publ. Math. Inst. Hautes \'{E}tudes Sci.}, (113):97--208, 2011.

\bibitem[KR08]{MR2391017}
Mikhail Khovanov and Lev Rozansky.
\newblock Matrix factorizations and link homology.
\newblock {\em Fund. Math.}, 199(1):1--91, 2008.

\bibitem[Lip16]{MR3617634}
Robert Lipshitz.
\newblock Heegaard {F}loer homologies.
\newblock In {\em Lectures on quantum topology in dimension three}, volume~48
  of {\em Panor. Synth\`eses}, pages 131--174. Soc. Math. France, Paris, 2016.

\bibitem[LS19]{lipshitz2019khovanov}
Robert Lipshitz and Sucharit Sarkar.
\newblock Khovanov homology also detects split links.
\newblock arXiv:1910.04246, 2019.

\bibitem[LXZ20]{li2020detection}
Zhenkun Li, Yi~Xie, and Boyu Zhang.
\newblock Two detection results of {K}hovanov homology on links.
\newblock arXiv:2005.05897, 2020.

\bibitem[Mar20]{martin2020khovanov}
Gage Martin.
\newblock Khovanov homology detects ${T}(2,6)$.
\newblock arXiv:2005.02893, 2020.

\bibitem[Miy98]{MR1451821}
Katura Miyazaki.
\newblock Band-sums are ribbon concordant to the connected sum.
\newblock {\em Proc. Amer. Math. Soc.}, 126(11):3401--3406, 1998.

\bibitem[Ni14]{MR3294567}
Yi~Ni.
\newblock Homological actions on sutured {F}loer homology.
\newblock {\em Math. Res. Lett.}, 21(5):1177--1197, 2014.

\bibitem[ORS13]{MR3071132}
Peter Ozsv\'{a}th, Jacob Rasmussen, and Zolt\'{a}n Szab\'{o}.
\newblock Odd {K}hovanov homology.
\newblock {\em Algebr. Geom. Topol.}, 13(3):1465--1488, 2013.

\bibitem[OS04a]{MR2065507}
Peter Ozsv\'{a}th and Zolt\'{a}n Szab\'{o}.
\newblock Holomorphic disks and knot invariants.
\newblock {\em Adv. Math.}, 186(1):58--116, 2004.

\bibitem[OS04b]{MR2113020}
Peter Ozsv\'{a}th and Zolt\'{a}n Szab\'{o}.
\newblock Holomorphic disks and three-manifold invariants: properties and
  applications.
\newblock {\em Ann. of Math. (2)}, 159(3):1159--1245, 2004.

\bibitem[OS05]{MR2141852}
Peter Ozsv\'{a}th and Zolt\'{a}n Szab\'{o}.
\newblock On the {H}eegaard {F}loer homology of branched double-covers.
\newblock {\em Adv. Math.}, 194(1):1--33, 2005.

\bibitem[OS06]{MR2249248}
Peter Ozsv\'{a}th and Zolt\'{a}n Szab\'{o}.
\newblock Lectures on {H}eegaard {F}loer homology.
\newblock In {\em Floer homology, gauge theory, and low-dimensional topology},
  volume~5 of {\em Clay Math. Proc.}, pages 29--70. Amer. Math. Soc.,
  Providence, RI, 2006.

\bibitem[OS08]{MR2377279}
Peter Ozsv\'{a}th and Zolt\'{a}n Szab\'{o}.
\newblock Knot {F}loer homology and integer surgeries.
\newblock {\em Algebr. Geom. Topol.}, 8(1):101--153, 2008.

\bibitem[Ras15]{MR3447099}
Jacob Rasmussen.
\newblock Some differentials on {K}hovanov-{R}ozansky homology.
\newblock {\em Geom. Topol.}, 19(6):3031--3104, 2015.

\bibitem[RW20]{MR4164001}
Louis-Hadrien Robert and Emmanuel Wagner.
\newblock A closed formula for the evaluation of foams.
\newblock {\em Quantum Topol.}, 11(3):411--487, 2020.

\bibitem[Sar11]{MR2812456}
Sucharit Sarkar.
\newblock A note on sign conventions in link {F}loer homology.
\newblock {\em Quantum Topol.}, 2(3):217--239, 2011.

\bibitem[Wan20]{wang2020cosmetic}
Joshua Wang.
\newblock The cosmetic crossing conjecture for split links.
\newblock arXiv:2006.01070, 2020.

\bibitem[Wan21a]{MR4311818}
Joshua Wang.
\newblock Link {F}loer homology also detects split links.
\newblock {\em Bull. Lond. Math. Soc.}, 53(4):1037--1044, 2021.

\bibitem[Wan21b]{wang2021sln}
Joshua Wang.
\newblock On sl({N}) link homology with mod {N} coefficients.
\newblock arXiv:2111.02287, 2021.

\bibitem[Wed19]{MR3982970}
Paul Wedrich.
\newblock Exponential growth of colored {HOMFLY}-{PT} homology.
\newblock {\em Adv. Math.}, 353:471--525, 2019.

\bibitem[XZ20]{xie2020links}
Yi~Xie and Boyu Zhang.
\newblock On links with {K}hovanov homology of small ranks.
\newblock arXiv:2005.04782, 2020.

\bibitem[XZ21]{xie2021classification}
Yi~Xie and Boyu Zhang.
\newblock Classification of links with {K}hovanov homology of minimal rank.
\newblock arXiv:1909.10032, 2021.

\bibitem[Zem19]{Zem19a}
Ian Zemke.
\newblock Knot {F}loer homology obstructs ribbon concordance.
\newblock {\em Ann. of Math. (2)}, 190(3):931--947, 2019.

\end{thebibliography}
\bibliographystyle{alpha}

\vspace{10pt}

\textit{Department of Mathematics}

\textit{Harvard University}

\textit{Science Center, 1 Oxford Street}

\textit{Cambridge, MA 02138}

\textit{USA}

\vspace{10pt}

\textit{Email:} \texttt{jxwang@math.harvard.edu}

\end{document}